\documentclass[a4paper,11pt]{amsart}

\usepackage{amsmath, amsthm, amsfonts, amssymb}
\usepackage[bookmarks=false]{hyperref}
\usepackage{enumerate}

\usepackage{color}

\numberwithin{equation}{section}

\newtheorem{letterthm}{Theorem}

\newtheorem{lettercor}[letterthm]{Corollary}

\newtheorem{letterconj}[letterthm]{Conjecture}

\newtheorem{thm}{Theorem}[section]
\newtheorem{lem}[thm]{Lemma}

\newtheorem{prop}[thm]{Proposition}

\theoremstyle{definition}

\newcommand{\N}{\mathbf{N}}

\newcommand{\cR}{\mathcal{R}}
\newcommand{\cS}{\mathcal{S}}

\newcommand{\rL}{\mathord{\text{\rm L}}}

\newcommand{\rd}{\mathord{\text{\rm d}}}

\newcommand{\ovt}{\mathbin{\overline{\otimes}}}

\newcommand{\II}{{\rm II}}
\newcommand{\III}{{\rm III}}

\begin{document}

\title{Stability of products of equivalence relations}

\begin{abstract}
An ergodic p.m.p.\ equivalence relation $\cR$ is said to be stable if $\cR \cong \cR \times \cR_0$ where $\cR_0$ is the unique hyperfinite ergodic type $\II_1$ equivalence relation. We prove that a direct product $\cR \times \cS$ of two ergodic p.m.p.\ equivalence relations is stable if and only if one of the two components $\cR$ or $\cS$ is stable. This result is deduced from a new local characterization of stable equivalence relations. The similar question on McDuff $\II_1$ factors is also discussed and some partial results are given.
\end{abstract}

\address{Laboratoire de Math\'ematiques d'Orsay\\ Universit\'e Paris-Sud\\ CNRS\\ Universit\'e Paris-Saclay\\ 91405 Orsay\\ FRANCE}

\author{Amine Marrakchi}
\email{amine.marrakchi@math.u-psud.fr}

\thanks{The author is supported by ERC Starting Grant GAN 637601}



\maketitle


\section{Introduction}
An ergodic type $\II_1$ equivalence relation $\mathcal{R}$ is \emph{stable} if $\mathcal{R} \cong \cR \times \cR_0$ where $\cR_0$ is the unique hyperfinite ergodic type $\II_1$ equivalence relation. This notion was introduced and studied in \cite{JS87}, by analogy with its von Neumann algebraic counterpart \cite{McD69}. In particular, the following characterization of stability was obtained (see Page 3 for the notations): an ergodic type $\II_1$ equivalence $\cR$ is stable if and only if for every finite set $K \subset [[ \cR]]$ and every $\varepsilon > 0$, there exists $v \in [[\cR]]$ such that $v^2=0$, $vv^*+v^*v=1$ and
$$ \forall u \in K, \quad  \| vu-uv \|_2 < \varepsilon.$$
Our first theorem strengthens this characterization by showing that the condition $vv^*+v^*v=1$ can be removed, thus allowing $v$ to be arbitrarily small. 
\begin{letterthm} \label{local_stable}
An ergodic type $\II_1$ equivalence relation $\cR$ is stable if and only if for every finite set $K \subset [[\cR]]$ and every $\varepsilon > 0$, there exists $v \in [[\cR]]$ such that $v^2=0$ and
$$ \forall u \in K, \quad  \| vu-uv \|_2 < \varepsilon \|v\|_2. $$
\end{letterthm}
As an application of Theorem \ref{local_stable}, we obtain the following rigidity result:
\begin{letterthm} \label{main_eq}
Let $\cR$ and $\cS$ be two ergodic type $\II_1$ equivalence relations. Then the product equivalence relation $\cR \times \cS$ is stable if and only if $\cR$ is stable or $\cS$ is stable.
\end{letterthm}

As we said before, the study of stable equivalence relations was inspired by its von Neumann algebraic counterpart: the so-called \emph{McDuff} property. Recall that a $\II_1$ factor $M$ is called McDuff if $M \cong M \ovt R$ where $R$ is the hyperfinite $\II_1$ factor. In \cite{McD69}, it is shown that a $\II_1$ factor $M$ is McDuff if and only if for every finite set $K \subset M$ and every $\varepsilon > 0$, there exists $v \in M$ such that $v^2=0$, $vv^*+v^*v=1$ and
$$ \forall a \in K, \quad \|va-av\|_2 < \varepsilon.$$
Similarly to the equivalence relation case, we can strengthen this characterization by removing the condition $vv^*+v^*v=1$ and we obtain the following analog of Theorem \ref{local_stable}.
\begin{letterthm} \label{local_mcduff}
A type $\II_1$ factor $M$ is McDuff if and only if for every finite set $K \subset M$ and every $\varepsilon > 0$, there exists $x \in M$ such that $x^2=0$ and
$$\forall a \in K, \quad \| xa-ax\|_2 < \varepsilon \| x \|_2. $$
\end{letterthm}

In regard of this result and the similarity between the theory of stable equivalence relations and the theory of McDuff factors, we strongly believe that the following analog of Theorem \ref{main_eq} shoud be true.

\begin{letterconj} \label{main}
Let $M$ and $N$ be type $\II_1$ factors. Then $M \ovt N$ is McDuff if and only if $M$ is McDuff or $N$ is McDuff.
\end{letterconj}

Even though the proof of Theorem \ref{main_eq} does not admit a straightforward generalization to the von Neumann algebraic case, we can still provide some partial solutions to Conjecture \ref{main} by using a different approach. The first one strengthens \cite[Theorem 2.1]{WY14}.
\begin{letterthm} \label{abelian}
Let $M$ be a non-McDuff $\II_1$ factor and suppose that there exists an abelian subalgebra $A \subset M$ such that $M_\omega \subset A^\omega$. Then for every $\II_1$ factor $N$ we have that $M \ovt N$ is McDuff if and only if $N$ is McDuff.
\end{letterthm} 
Up to the author's knowledge, all concrete examples of non-McDuff factors in the literature do satisfy the assumption of Theorem \ref{abelian} (in fact, this is how we show that they are not McDuff). Deciding whether or not this property holds for all non-McDuff factors is an interesting open question.

The second result solves Conjecture \ref{main} under the additional assumption that $(M \ovt N)_\omega$ is a factor.

\begin{letterthm} \label{factor}
Let $M$ and $N$ be type $\II_1$ factors and suppose that $(M \ovt N)_\omega$ is a factor. Then both $M_\omega$ and $N_\omega$ are factors. In particular, if $M \ovt N$ is McDuff, then $M$ is McDuff or $N$ is McDuff.
\end{letterthm}

Theorem \ref{factor} also has an interesting application which is not related to Conjecture \ref{main}. It provides examples of McDuff $\II_1$ factors that do not admit any McDuff decomposition in the sense of \cite{HMV16}. This result uses and strengthens \cite[Theorem 4.1]{Po10}.

\begin{lettercor} \label{rigid}
Let $M=\overline{\bigotimes}_{n \in \N} M_n$ be an infinite tensor product of non-Gamma type $\II_1$ factors $M_n, \: n \in \N$. Let $N$ be a type $\II_1$ factor such that $ M \cong N \ovt R$. Then $M \cong N$.
\end{lettercor}

Before we end this introduction, let us say a few words about the methods used to obtain these results. The proof of Theorem \ref{local_stable} (and Theorem \ref{local_mcduff}) is based on a so-called \emph{maximality} argument. This technique consists in patching "microscopic" elements satisfying a given property in order to obtain a "macroscopic" element satisfying this same property. The name \emph{maximality} is a reference to Zorn's lemma which is used in the patching procedure. Maximality arguments in the theory of von Neumann algebras were initiated in \cite{MvN43}. Since then, they have been used fruitfully in many of the deepest results of the theory, reaching higher and higher levels of sofistication in \cite{Co75b}, \cite{CS76}, \cite{Co85}, \cite{Ha85}, \cite{Po85} and culminating in the \emph{incremental patching} method of \cite{Po87}, \cite{Po95}, \cite{Po14}. See also \cite{Ma16}, \cite{HMV17} and \cite{Ma17} for other recent applications of maximality arguments. On the other hand, the proofs of Theorem \ref{abelian} and Theorem \ref{factor} are based on a completely different technique which appears in \cite{IV15} and which is inspired by an averaging trick of Haagerup \cite{Ha84}. By using this technique, one can reduce some problems on arbitrary tensor products $M \ovt N$ to the much easier case where one of the two algebras is abelian. This very elementary transfer principle is surprisingly powerful and Theorem \ref{abelian} and Theorem \ref{factor} are two applications among many others.

\subsection*{Acknowledgments} It is our pleasure to thank Cyril Houdayer and Sorin Popa for their valuable comments. We also thank Yusuke Isono for the thought-provoking discussions we had.

\subsection*{Notations}
For simplicity, in this paper, all probability spaces are standard and all von Neumann algebras have separable predual (except ultraproducts). We fix some non-principal ultrafilter $\omega \in \beta \N \setminus \N$ once and for all. We denote by $\rL(\cR)$ the von Neumann algebra of a p.m.p.\ equivalence relation $\cR$ (see \cite{FM75}).We denote by $[\cR]$ (resp.\ $[[\cR]]$) its full group (resp.\ full pseudo-group) and we will identify them with the corresponding unitaries (resp.\ partial isometries) in the von Neumann algebra $\rL(\cR)$. In particular if $v,w \in [[\cR]]$, then $\|v-w\|_2$ refers to the $2$-norm of $\rL(\cR)$.

\tableofcontents

\section{A local characterization of stable equivalence relations}

In this section, we establish the following more precise version of Theorem \ref{local_stable}. The proof is inspired by \cite[Theorem 2.1]{Co75b} and \cite[Theorem 2]{Co85}.

\begin{thm} \label{local_stable_det}
Let $\cR$ be an ergodic p.m.p.\ equivalence relation on a probability space $(X,\mu)$. Then the following are equivalent:
\begin{itemize}
\item [$(\rm i)$] $\cR$ is stable.
\item [$(\rm ii)$]  For every finite set $K \subset [[\cR]]$ and every $\varepsilon > 0$, there exists $v \in [[\cR]]$ such that 
$$ v^2=0,$$
$$vv^*+v^*v=1,$$
$$ \forall u \in K, \quad  \| vu-uv \|_2 < \varepsilon. $$
\item [$(\rm iii)$]  For every finite set $K \subset [[\cR]]$ and every $\varepsilon > 0$, there exists $v \in [[\cR]]$ such that
$$ v^{2}=0,$$
$$ \forall u \in K, \quad  \| vu-uv \|_2 < \varepsilon \|v\|_2. $$
\item [$(\rm iv)$]  For every finite set $K \subset [[\cR]]$ and every $\varepsilon > 0$, there exists $v \in [[\cR]]$ such that
$$ \forall u \in K, \quad  \| vu-uv \|_2 < \varepsilon \|vv^*-v^*v\|_2. $$
\end{itemize}
\end{thm}
\begin{proof}
The equivalence $(\rm i) \Leftrightarrow (\rm ii)$ is proved in \cite{JS87}. The implications $ (\rm ii) \Rightarrow (\rm iii) \Rightarrow (\rm iv)$ are clear.

First, we show that $(\rm iv) \Rightarrow (\rm iii)$. Let $K=K^* \subset [[\cR]]$ be a finite symmetric set and $\varepsilon > 0$. Choose $v \in [[\cR]]$ such that 
$$ \forall u \in K, \quad  \| vu-uv \|_2 < \varepsilon \|vv^*-v^*v\|_2. $$
Let $w=v(1-vv^*) \in [[\cR]]$. Then $w^2=0$. Moreover, a simple computation shows that  
$$ \forall u \in K, \quad  \| wu-uw \|_2  < 3\varepsilon \|vv^*-v^*v\|_2$$
and we have $\|vv^*-v^*v\|_2=\|ww^*-w^*w\|_2=\sqrt{2}\|w\|_2$. Hence we obtain
$$ \forall u \in K, \quad  \| wu-uw \|_2  < 3\sqrt{2}\varepsilon \|w\|_2.$$

Next, we prove that if $\cR$ satisfies condition $(\rm iii)$ then every corner of $\cR$ also satisfies it. Let $Y \subset X$ be a non-zero subset and $p=1_Y$. Suppose that the corner $\cR_Y$ doest not satisfy $(\rm iii)$. Then we can find a finite set $K \subset p[[\cR]]p$ and a constant $\kappa > 0$ such that for all $v \in p[[\cR]]p$ with $v^2=0$, we have
\[ \| v\|_2^2 \leq \kappa \sum_{u \in K} \|vu-uv\|_2^2. \]
Since $\cR$ is ergodic, we can find a finite set $S \subset [[\cR]]$ such that $\sum_{w \in S} w^*w=p^\perp$ and $ww^* \leq p$ for all $w \in S$. Then, for every $v \in [[\cR]]$, we have
\[ \|v\|_2^2=\|pv\|_2^2+\sum_{ w \in S} \|wv\|_2^2\]
and for all $w \in S$, we have
\[ \|wv\|_2^2 \leq 2( \|wv-vw\|_2^2+\|vw\|_2^2) \leq 2( \|wv-vw\|_2^2+\|vp\|_2^2).\]
Hence, we obtain
\[ \|v\|_2^2 \leq 2 |S| (\|vp\|_2^2+\|pv\|_2^2) + 2\sum_{w \in S} \|wv-vw\|_2^2. \]
Moreover, we have
\[ \|pv\|_2^2+\|vp\|_2^2 = \|pv-vp\|_2^2 + 2 \|pvp\|_2^2, \]
hence
\[ \|v\|_2^2 \leq 2|S| \|pv-vp\|_2^2 +4 |S| \|pvp\|_2^2+ 2\sum_{w \in S} \|wv-vw\|_2^2. \]
Since $(pvp)^2 =0$, we know, by assumption, that
\[ \|pvp\|_2^2 \leq \kappa \sum_{u \in K} \| (pvp)u-u(pvp)\|_2^2 \leq \kappa \sum_{u \in K} \|vu-uv\|_2^2.\]
Therefore, we finally obtain
\[ \|v\|_2^2 \leq 2 |S| \|pv-vp\|_2^2 + 4|S| \kappa \sum_{u \in K} \|vu-uv\|_2^2 + 2 \sum_{w \in S} \|wv-vw\|_2^2. \]
This shows that $\cR$ does not satisfy $(\rm iii)$.

Finally, we use a maximality argument to show that $(\rm iii) \Rightarrow (\rm ii)$. Let $u_1,\dots,u_n \in [[\cR]]$ be a finite family and let $\varepsilon> 0$ and $\delta =8 \varepsilon$. Consider the set $ \Lambda$ of all $(v,U_1,\dots,U_n) \in [[\cR]]^{n+1}$ such that
\begin{itemize}
\item $v^2=0$.
\item $[U_k,vv^{*}+v^{*}v]=0$ for all $k=1,\dots,n$.
\item $ \|vU_k-U_kv\|_2 \leq \varepsilon \| v \|_2$ for all $k=1,\dots,n$.
\item $\| U_k-u_k\|_1 \leq \delta \|v\|_1$.
\end{itemize}
On $\Lambda$ put the order relation given by
\[ (v,U_1,\dots,U_n) \leq (v',U'_1,\dots,U'_n) \]
if and only if $v \leq v'$ and $\|U_k'-U_k\|_1\leq \delta ( \|v'\|_1-\|v\|_1)$ for all $k=1,\dots,n$. Then $\Lambda$ is an inductive set (because $[[\mathcal{R}]]$ is inductive and is also complete for the distance given by $\| \cdot \|_1$). By Zorn's lemma, let $v \in \Lambda$ be a maximal element. Suppose that $q=vv^*+v^*v \neq 1$. Since, by the previous step, all corners of $\cR$ also satisfy $(\rm iii)$, we can find a non-zero element $w \in q^{\perp}[[\cR]] q^{\perp}$, with $w^2=0$ such that
\[ \| wU_k-U_kw \|_2 \leq \varepsilon \|w\|_2 \]
\[\| wU_k^*-U_k^*w \|_2 \leq \varepsilon \|w\|_2 \]
for all $k=1, \dots, n$. 

Now, let 
\begin{itemize}
\item $p:=ww^*+w^*w$
\item $U'_k:=pU_kp + p^{\perp}U_kp^{\perp}$
\item $v':=v+w$
\item $q':=v'(v')^*+(v')^*v'=q+p$
\end{itemize}
Note that $(v')^2=0$ and $[U_k',q']=0$ for all $k$. We also have
\[ \|v'U'_k-U'_kv'\|_2^2 \leq \|vU_k-U_kv\|_2^2+\|wU_k-U_kw\|_2^2 \leq \varepsilon^2 \|v\|_2^2+\varepsilon^2 \|w\|_2^2=\varepsilon^2 \|v'\|_2^2. \]
Moreover, by Cauchy-Schwarz inequality, we have
\begin{align*}
\| U_k'-U_k \|_1  & \leq \| pU_kp^{\perp}\|_1 + \|p^{\perp}U_kp\|_1 \\
& \leq \|p\|_2(\| p U_kp^{\perp}\|_2+\|p^{\perp}U_kp\|_2) \\
& \leq \sqrt{2} \|p\|_2 \| [U_k,p] \|_2 \\
& \leq 2 \sqrt{2}\|p\|_2 ( \| [U_k,w] \|_2 + \| [U_k,w^* ] \|_2) \\
& \leq 4 \sqrt{2} \varepsilon \|p\|_2 \|w\|_2 \\
& = 8\varepsilon \|w\|_2^2 \\
& = \delta \|w\|_1.
\end{align*} 
Since $\|v'\|_1=\|v\|_1+\|w\|_1$, this implies that
\[ \|U'_k-U_k \|_1 \leq \delta ( \|v'\|_1-\|v\|_1) \]
and
\[ \|U'_k-u_k \|_1 \leq \| U'_k-U_k\|_1+\|U_k-u_k\|_1 \leq \delta \|v'\|_1.\]
Therefore $v' \in \Lambda$ and $v \leq v'$. This contradicts the maximality of $v$. Hence we must have $v^*v+vv^*=q=1$. Moreover, since
\[ \|vu_k-u_kv\|_2 \leq \|vU_k-U_k v \|_2 + 2\|U_k-u_k\|_2,\]
\[ \|vU_k-U_k v \|_2 \leq \varepsilon \]
and
\[ \|U_k-u_k\|_2^2 \leq 2 \|U_k-u_k\|_1 \leq 2 \delta=16 \varepsilon, \]
we conclude that
\[ \| vu_k-u_k v \|_2 \leq  \varepsilon +8 \sqrt{\varepsilon}. \]
Since such a $v$ exists for every $\varepsilon > 0$, we have proved $( \mathrm{ii})$.
\end{proof}

\section{Proof of Theorem \ref{main_eq}}
In this section, we prove Theorem \ref{main_eq}. We need to introduce some notations which will be useful in order to decompose elements of the full pseudo-group $[[\cR \times \cS]]$ as functions from $\cR$ to $[[\cS]]$.

Let $\cR$ be a p.m.p.\ equivalence relation on a probability space $(X,\mu)$. We denote by $\tilde{\mu}$ the canonical $\sigma$-finite measure on $\cR$ induced by $\mu$. Then $\rL^2(\cR):=\rL^2(\cR,\tilde{\mu})$ can be identified with the canonical $\rL^2$-space of $\rL(\cR)$. For every $x \in \rL(\cR)$, we denote by $\widehat{x}$ the corresponding vector in $\rL^2(\cR)$. If $v \in [[\cR]]$, then $\widehat{v}$ is just the indicator function of the graph of $v$. We denote by $\mathfrak{P}(X)$ the set of projections of $\rL^\infty(X,\mu)$. For every $p \in \mathfrak{P}(X)$, we can view $\widehat{p}$ as an indicator function in $\rL^2(X)$, where $\rL^2(X)$ is embedded into $\rL^2(\cR)$ via the diagonal inclusion.

If $\mathcal{S}$ is a second p.m.p.\ equivalence relation on $(Y,\nu)$, then for any $v \in [[\cR \times \cS]]$, there exists a unique function $v_{\cS} \in \rL^0(\cR, [[\cS]])$ which satisfies
$$ \widehat{v}(x,x',y,y')=\widehat{v_\cS(x,x')}(y,y')$$
for a.e.\ $(x,x',y,y') \in \cR \times \cS$. 

If $p \in \mathfrak{P}(X \times Y)$, then there exists a unique function $p_Y \in \rL^0(X,\mathfrak{P}(Y))$ such that
$$ \widehat{p}(x,y')=\widehat{p_Y(x)}(y)$$
for a.e.\ $(x,y) \in X \times Y$. 

All these heavy notations are needed for the following key lemma which allows us to decompose a commutator in $[[\cR \times \cS]]$ into two parts that we will be able to control independently. The proof is just an easy computation.

\begin{lem} \label{split}
Let $\cR$ and $\cS$ be two p.m.p.\ equivalence relations on $(X,\mu)$ and $(Y,\nu)$ respectively. Let $\cR \times \cS$ be the product p.m.p.\ equivalence relation on $(X \times Y, \mu \otimes \nu)$. Let $v \in [[\cR \times \cS]]$ and $p \in \mathfrak{P}(X \times Y)$. Let $v_1:=v_\cR \in \rL^0(\cS, [[\cR]])$ and $v_2:=v_\cS \in \rL^0(\cR, [[\cS]])$ the two functions defined by $v$. Let $p_1:=p_X \in \rL^0(Y,\mathfrak{P}(X))$ and $p_2:=p_Y \in \rL^0(X,\mathfrak{P}(Y))$ the two functions defined by $p$.

Define $\xi_1 \in \rL^2(\cS, \rL^2(\cR))$ by
\[ \xi_1(y,y')=\widehat{[v_1(y,y'),p_1(y)]}, \quad  \text{for a.e. } (y,y') \in \cS. \]
and $\xi_2 \in \rL^2(\cR,\rL^2(\cS))$ by
\[ \xi_2(x,x')=\widehat{[v_2(x,x'),p_2(x')]}, \quad  \text{for a.e. } (x,x') \in \cR. \]

Then, after identifying $\rL^2(\cS, \rL^2(\cR)) \cong \rL^2(\cR,\rL^2(\cS))\cong \rL^{2}(\cR \times \cS)$, we have $\widehat{[v,p]}=\xi_1+\xi_2$.
\end{lem}
\begin{proof}
For a.e.\ $(x,x',y,y') \in \cR \times \cS$, we compute
$$ (\xi_1(y,y'))(x,x')=\widehat{v}(x,x',y,y')(\widehat{p}(x,y)-\widehat{p}(x',y))$$
$$ (\xi_2(x,x'))(y,y')=\widehat{v}(x,x',y,y')(\widehat{p}(x',y)-\widehat{p}(x',y'))$$
$$ \widehat{[v,p]}(x,x',y,y')=\widehat{v}(x,x',y,y')(\widehat{p}(x,y)-\widehat{p}(x',y'))$$
hence the equality we want.
\end{proof}

\begin{proof}[Proof of Theorem \ref{main_eq}]
Clearly, if $\cR$ or $\cS$ is stable then $\cR \times \cS$ is also stable. Now, suppose that $\cR$ and $\cS$ are not stable. Then, by Theorem \ref{local_stable_det}, we can find a constant $\kappa_1 > 0$ and a finite set $K_1 \subset [[\cR]]$ such that for all $v \in [[\cR]]$ we have
\[ \| vv^*-v^*v \|_2^2 \leq \kappa_1 \sum_{u \in K_1} \|vu-uv\|_2^2. \]
Similarly, we can find a constant $\kappa_2 > 0$ and a finite set $K_2 \subset [[\cS]]$ such that for all $v \in [[\cS]]$ we have
\[ \| vv^*-v^*v \|_2^2 \leq \kappa_2 \sum_{u \in K_2} \|vu-uv\|_2^2. \]
In order to prove that $\cR \times \cS$ is not stable, we will show that for all $v \in [[\cR \times \cS]]$ with $v^2=0$, we have
\[ \| v \|_2^2 \leq \kappa \sum_{u \in K } \|vu-uv\|_2^2 \]
where $\kappa=2(\kappa_1+\kappa_2)$ and $K=(K_1 \otimes 1) \cup (1 \otimes K_2)$.

Indeed, let $v \in [[\cR \times \cS]]$ with $v^2=0$ and let $p=v^*v$. Using the notations of Lemma \ref{split}, we can write $\widehat{v}=\widehat{[v,p]}=\xi_1+\xi_2$ and we have the formulas
\[ \|\xi_1\|_2^2=\int_{\cS} \| v_1(y,y')p_1(y)-p_1(y)v_1(y,y') \|_2^2 \; \rd \nu_\ell(y,y'), \]
\[ \|\xi_2\|_2^2=\int_{\cR} \| v_2(x,x')p_2(x')-p_2(x')v_2(x,x') \|_2^2 \; \rd \mu_\ell(x,x'). \]
Since $pv=0$, then for a.e.\ $(y,y') \in \cS$ we have that $p_1(y)v_1(y,y')=0$, hence
\[ v_1(y,y')p_1(y)-p_1(y)v_1(y,y')=v_1(y,y')\left(v_1(y,y')^*v_1(y,y')-v_1(y,y')v_1(y,y')^*\right)p_1(y).  \]
This shows that
\begin{align*}
 \| v_1(y,y')p_1(y)-p_1(y)v_1(y,y') \|_2^2  & \leq \| v_1(y,y')^*v_1(y,y')-v_1(y,y')v_1(y,y')^*\|_2^2 \\ &\leq \kappa_1 \sum_{u \in K_1} \| v_1(y,y')u-uv_1(y,y')\|_2^2
\end{align*}
After integrating over $\cS$ and using the following formula 
\[ \forall u \in K_1, \; \| v(u \otimes 1)-(u \otimes 1)v\|_2^2=\int_{\cS} \| v_1(y,y')u-uv_1(y,y') \|_2^2 \; \rd \nu_\ell(y,y'), \]
we obtain
\[ \|\xi_1\|_2^2 \leq \kappa_1 \sum_{u \in K_1} \| v(u \otimes 1)-(u \otimes 1)v\|_2^2. \]

Similarly, since $vp=v$, we have $v_2(x,x')p_2(x')=v_2(x,x')$ for a.e\ $(x,x') \in \cR$, hence
\[ v_2(x,x')p_2(x')-p_2(x')v_2(x,x')=p_2(x')^{\perp}\left(v_2(x,x')v_2(x,x')^*-v_2(x,x')^*v_2(x,x')\right)v_2(x,x').  \]
Then, proceeding as before, one shows that
\[ \|\xi_2\|_2^2 \leq \kappa_2 \sum_{u \in K_2} \| v(1 \otimes u)-(1 \otimes u)v\|_2^2. \]
Finally, since $\widehat{v}=\widehat{[v,p]}=\xi_1+\xi_2$, we conclude that
\[ \|v\|_2^2 \leq 2 (\| \xi_1 \|_2^2+ \|\xi_2\|_2^2) \leq \kappa \sum_{u \in K} \|vu-uv\|_2^2 \]
as we wanted.
\end{proof}

\section{A local characterization of McDuff factors}

In this section, we establish Theorem \ref{local_mcduff}. The proof is more involved than the proof of Theorem \ref{local_stable}. We will need the following lemma (for a proof see \cite[Lemma 1.2.6]{Co75b}, \cite[Proposition 1]{CS76} and \cite[Theorem 2]{CS76}).

\begin{lem} \label{integrate}
Let $M$ be a von Neumann algebra. For every $x \in M$ and every $t \geq 0$, let $$u_t(x)=u1_{[t,+\infty)}(|x|)$$ where $x=u|x|$ is the polar decomposition of $x$.
\begin{enumerate}

\item For all $x \in M$ we have 
$$ \int_0^\infty \| u_{t^{1/2}}(x)\|_2^2 \; \rd t = \|x \|_2^2.$$
\item For all  $x,y \in M^+$, we have
$$\|x-y\|_2^2 \leq \int_0^\infty \| u_{t^{1/2}}(x)-u_{t^{1/2}}(y)\|_2^2 \; \rd t .$$ 
\item For all  $x \in M$ and all $a \in M^+$, we have
$$ \int_0^\infty \| u_{t^{1/2}}(x)a-au_{t^{1/2}}(x)\|_2^2 \; \rd t \leq 4 \|xa-ax\|_2 \| xa+ax\|_2.$$ 
\end{enumerate}
\end{lem}

Now, we can prove the following more precise version of Theorem \ref{local_mcduff}. Note that even if one is only interested in item $(\rm iii)$, one still needs first to prove that it is equivalent to $(\rm iv)'$.

\begin{thm} \label{local_mcduff_det}
Let $M$ be a factor of type $\II_1$ with separable predual. Then the following are equivalent:
\begin{itemize}
\item [$(\rm i)$] $M$ is McDuff.
\item [$(\rm ii)$]  For every finite set $F \subset M$ and every $\varepsilon > 0$, there exists a partial isometry $v \in M$ such that 
$$vv^*+v^*v=1,$$
$$ \forall a \in F, \quad  \| va-av \|_2 < \varepsilon. $$
\item [$(\rm iii)$]  For every finite set $F \subset M$ and every $\varepsilon > 0$, there exists a partial isometry $v \in M$ such that
$$v^2=0,$$
$$ \forall a \in F, \quad  \| va-av \|_2 < \varepsilon \|v\|_2.$$
\item [$(\rm iii)'$]  For every finite set $F \subset M$ and every $\varepsilon > 0$, there exists $x \in M$ such that
$$ x^{2}=0,$$
$$ \forall a \in F, \quad  \| xa-ax \|_2 < \varepsilon \|x\|_2. $$
\item [$(\rm iv)$]  For every finite set $F \subset M$ and every $\varepsilon > 0$, there exists a partial isometry $v \in M$ such that
$$ \forall a \in F, \quad  \| va-av \|_2 < \varepsilon \|vv^*-v^*v\|_2.$$
\item [$(\rm iv)'$]  For every finite set $F \subset M$ and every $\varepsilon > 0$, there exists $x \in M$ such that
$$ \forall a \in F, \quad  \|x\|_2\cdot \| xa-ax \|_2 < \varepsilon \||x|-|x^*|\|_2^2. $$
\end{itemize}
\end{thm}
\begin{proof}
The equivalence $(\rm i) \Leftrightarrow (\rm ii)$ is already known \cite{McD69}. First we show that $(\rm iii) \Leftrightarrow (\rm iii)' \Leftrightarrow (\rm iv) \Leftrightarrow (\rm iv)'$. For this, we will prove the following implications $(\rm iii) \Rightarrow (\rm iv)' \Rightarrow (\rm iv) \Rightarrow (\rm iii)' \Rightarrow (\rm iii)$.

$(\rm iii) \Rightarrow (\rm iv)'$. If $v$ satisfies $(\rm iii)$ then $x:=v$ also satisfies $(\rm iv)'$ since $ \| |x|-|x^*| \|_2 =\sqrt{2} \|x\|_2.$

$(\rm iv)' \Rightarrow (\rm iv)$. Suppose, by contradiction, that there exists a finite set $F \subset M$ and a constant $\kappa > 0$ such that for all partial isometries $v \in M$ we have
\[ \|vv^{*}-v^{*}v \|_2^2 \leq \kappa \sum_{ a \in F} \| va-av\|_2^2. \]
We can assume that $F \subset M^+$. Let $x \in M$. Then by Lemma \ref{integrate}, we have
\[ \| |x^*|-|x|\|_2^2 \leq \int_0^\infty \| u_{t^{1/2}}(|x^*|)-u_{t^{1/2}}(|x|) \|_2^2 \; \rd t \leq \kappa \sum_{a \in F} \int_0^\infty \|u_{t^{1/2}}(x)a-au_{t^{1/2}}(x) \|_2^2 \; \rd t. \] 
Since, for every $a \in F$, we have
\[ \int_0^\infty \|u_{t^{1/2}}(x)a-au_{t^{1/2}}(x) \|_2^2 \; \rd t \leq 4 \| x a-ax\|_2  \|xa+ax\|_2 \leq 8\| a\|_\infty \|x\|_2  \|xa-ax\|_2, \]
we obtain
\[ \| |x^*|-|x|\|_2^2  \leq 8\kappa  \left( \max_{a \in F} \|a\|_\infty\right) \|x\|_2 \sum_{a \in F} \|xa-ax\|_2 \]
and this contradicts $(\rm iv)'$.

$(\rm iv) \Rightarrow (\rm iii)'$. Let $F \subset M$ be a finite self-adjoint set and $\varepsilon > 0$. Pick $v \in M$ a partial isometry such that
$$ \forall a \in F, \quad  \| va-av \|_2 < \varepsilon \|vv^*-v^*v\|_2.$$
Let $x_1=(1-v^{*}v)v$ and $x_2=v(1-vv^{*})$. Note that $x_1^2=x_2^2=0$. Let $x:=x_1$ if $\|x_1\| \geq \|x_2 \|$ and $x:=x_2$ otherwise. Then we have
$$ \|vv^{*}-v^{*}v\|_2^{2}=\|x_1\|_2^{2}+\|x_2\|_2^{2} \leq 2 \|x\|_2^2.$$
Moreover, since $F$ is self-adjoint, we have
$$ \forall a \in F, \quad \|xa-ax\|_2 \leq 3 \| va-av\|_2. $$
Therefore, we obtain
$$ \forall a \in F, \quad  \| xa-ax \|_2 < 3 \sqrt{2}\varepsilon \|x\|_2. $$

$(\rm iii)' \Rightarrow (\rm iii)$. Suppose, by contradiction, that there exists a finite set $F \subset M$ and a constant $\kappa > 0$ such that for all partial isometries $v \in M$ with $v^2=0$ we have
\[ \|v \|_2^2 \leq \kappa \sum_{ a \in F} \| va-av\|_2^2. \]
We can assume that $F \subset M^+$. Let $x \in M$ such that $x^2=0$. Then for every $t > 0$, we have $u_t(x)^2=0$. Hence, by Lemma \ref{integrate}, we have
\[ \|x\|_2^2=\int_0^\infty \| u_{t^{1/2}}(x) \|_2^2 \; \rd t \leq \kappa \sum_{a \in F} \int_0^\infty \|u_{t^{1/2}}(x)a-au_{t^{1/2}}(x) \|_2^2 \; \rd t. \] 
Since, for every $a \in F$, we have
\[ \int_0^\infty \|u_{t^{1/2}}(x)a-au_{t^{1/2}}(x) \|_2^2 \; \rd t \leq 4 \| x a-ax\|_2  \|xa+ax\|_2 \leq 8\| a\|_\infty \|x\|_2  \|xa-ax\|_2, \]
we obtain
\[ \|x \|_2 \leq 8 \left( \max_{a \in F} \|a\|_\infty\right) \kappa \sum_{a \in F} \|xa-ax\|_2 \]
and this contradicts $(\rm iii)'$.

This finishes the proof of the equivalences $(\rm iii) \Leftrightarrow (\rm iii)' \Leftrightarrow (\rm iv) \Leftrightarrow (\rm iv)'$. Next, we will prove that if $M$ satisfies $(\rm iii)$ then $pMp$ also satisfies $(\rm iii)$ for every non-zero projection $p \in M$. Suppose, by contradiction, $pMp$ does not satisfy $(\rm iii)$. Then $pMp$ does not satisfy $(\rm iv)'$. Hence we can find a constant $\kappa > 0$ and a finite set $F \subset pMp$ such that
$$ \forall x \in pMp, \quad \| |x|-|x^*| \|_2^2 \leq \kappa \|x\|_2 \sum_{a \in F} \| ax-xa \|_2. $$
Take $S \subset M$ a finite set of partial isometries such that $\sum_{w \in S} w^*w=p^\perp$ and $ww^* \leq p$ for all $w \in S$. Now, take a partial isometry $v \in M$ with $v^2=0$ and let $x:=pvp$. Then we have
$$ \| v\|_2^2 = \|pv\|_2^2 + \sum_{w \in S} \| wv\|_2^2$$
and for all $w \in S$, we have
$$ \|wv \|_2^2 \leq 2( \| wv-vw\|_2^2+\|vw\|_2^2) \leq 2(\|wv-vw\|_2^2+\|vp\|_2^2). $$
Hence, we obtain
\[ \|v\|_2^2 \leq 2 |S| (\|vp\|_2^2+\|pv\|_2^2) + 2\sum_{w \in S} \|wv-vw\|_2^2. \]
Moreover, we have
\[ \|pv\|_2^2+\|vp\|_2^2 = 2 \|x\|_2^2 + \|pv-vp\|_2^2 , \]
hence
\[ \|v\|_2^2 \leq 4 |S| \|x\|_2^2 +  2|S| \|pv-vp\|_2^2+ 2\sum_{w \in S} \|wv-vw\|_2^2. \]
Now, by assumption, we have
$$\| |x|-|x^*| \|_2^2 \leq \kappa \|x\|_2 \sum_{a \in F} \| ax-xa \|_2.$$
Moreover, by using the fact that $v^2=0$, it is not hard to check that 
$$ \|x\|_2 \leq \| |x|-|x^*| \|_2 + 3 \| pv-vp\|_2 $$
which implies that
$$ \|x\|_2^2 \leq 2 \kappa \|x\|_2 \sum_{a \in F} \|ax-xa\|_2 + 18 \| pv-vp\|_2^2. $$
Therefore we obtain
$$ \|v\|_2^2 \leq 8|S|\kappa \|x\|_2 \sum_{a \in F} \|ax-xa\|_2 + 74 |S| \|pv-vp\|_2^2 + 2 \sum_{w \in S} \|wv-vw\|_2^2.$$
Finally, using the fact that $$\|x\|_2 \sum_{a \in F} \|ax-xa \|_2 \leq \|v\|_2 \sum_{a \in F} \| av-va\|_2, $$ 
$$ \|pv-vp\|_2^2 \leq 2 \|v \|_2 \|pv-vp\|_2$$
and 
$$ \|wv-vw\|_2^2 \leq 2 \|v\|_2 \|wv-vw\|_2,$$
we can conclude that
$$ \|v\|_2 \leq \kappa' \sum_{a \in F'} \|av-va\|_2. $$
for some $\kappa' > 0$, some finite set $F' \subset M$ and all partial isometries $v \in M$ with $v^2=0$. This shows that $M$ does not satisfy $(\rm iii)$ as we wanted.

Finally, one can prove $(\rm iii) \Rightarrow (\rm ii)$ by using exactly the same maximality argument that we used in the proof of Theorem \ref{local_stable_det}.
\end{proof}

\section{Another approach to Question \ref{main}}
The following lemma is extracted from \cite{IV15} and it is inspired by a trick used in \cite{Ha84}. Recall that if $M$ is a von Neumann algebra, then $\rL^2(M^\omega)$ is in general much smaller than the ultraproduct Hilbert space $\rL^2(M)^\omega$ (see \cite[Proposition 1.3.1]{Co75b}).
\begin{lem} \label{Haagerup}
Let $M$ and $N$ be finite von Neumann algebras. Fix a tracial state $\tau$ on $M$ and pick an orthonormal basis $(e_n)_{n \in \N}$ of $(M, \tau)$. Let $A=\rL^\infty(\mathbb{T}^{\N})=\rL^\infty(\mathbb{T})^{\ovt \N}$ and for each $n \in \N$, let $u_n \in \mathcal{U}(A)$ be the canonical generator of the $n$th copy of $\rL^\infty(\mathbb{T})$. Let $V : \rL^2(M) \rightarrow \rL^2(A)$ be the unique (non-surjective) isometry which sends $e_n$ to $u_n$ for every $n \in \N$. 

Then the naturally defined ultraproduct isometry
$$ (V \otimes 1)^\omega : \rL^2(M \ovt N)^\omega \rightarrow \rL^2(A \ovt N)^\omega $$ sends $\rL^2((M \ovt N)^\omega)$ into $\rL^2((A \ovt N)^\omega)$.
\end{lem}

Lemma \ref{Haagerup} is useful because it allows us to reduce many problems on sequences in tensor products $M \ovt N$ to the case where $M$ is abelian. We now present two applications of this principle.

The first one slightly generalizes \cite[Corollary]{IV15}. We will need it for Theorem \ref{abelian}.
\begin{prop} \label{prop_central_1}
Let $M$ and $N$ be finite von Neumann algebras. Suppose that there exists two von Neumann subalgebras $Q,P \subset N$ such that $Q' \cap N^\omega \subset P^\omega$, then we have 
$$ (1 \otimes Q)' \cap (M \ovt N)^\omega \subset (M \ovt P)^\omega.$$
\end{prop}
\begin{proof}
First, we deal with the case where $M$ is abelian, i.e.\ $M=\rL^\infty(T,\mu)$ for some probability space $(T,\mu)$. Take $(x_n)^\omega$ in the unit ball of $(1 \ovt Q)' \cap (M \ovt N)^\omega$ and write $x_n=( t \mapsto x_n(t)) \in M \ovt N=\rL^\infty(T,\mu,N)$ for every $n \in \N$. Let $\varepsilon > 0$ and choose a finite set $F \subset Q$ and $\delta > 0$ such that for every $x$ in the unit ball of $N$ we have 
$$\left( \forall a \in F, \: \| [x,a]\|_2 \leq \delta \right) \Longrightarrow \|x-E_P(x) \|_2 \leq  \varepsilon.$$
Since $(x_n)^\omega \in (1 \ovt Q)' \cap (M \ovt N)^\omega$, we have
$$\lim_{n \rightarrow \omega} \mu \left(\{ t \in T \mid  \forall a \in F, \: \| [x_n(t),a]\|_2 \leq \delta \}\right)=1. $$
Hence, we have 
$$\lim_{n \rightarrow \omega} \mu \left(\{ t \in T \mid  \|x_n(t)-E_P(x_n(t)) \|_2 \leq  \varepsilon \}\right)=1. $$
This means that 
$$\lim_{n \rightarrow \omega} \| x_n - E_{M \ovt P}(x_n) \|_2 \leq \varepsilon$$
and since this holds for every $\varepsilon > 0$, we conclude that $(x_n)^\omega \in (M \ovt P)^\omega$.

Now, we extend to the general case where $M$ is not necessarily abelian. Let $\xi \in \rL^2((M \ovt N)^\omega)$ be an $Q$-central vector. We want to show that $\xi \in \rL^2((M \ovt P)^\omega)$. By Lemma \ref{Haagerup}, we know that $\eta=(V \otimes 1)^\omega (\xi) \in \rL^2((A \ovt N)^\omega)$. Since $(V \otimes 1)^\omega$ is $N$-bimodular, we know that $\eta$ is $Q$-central.  Hence, by the abelian case, we obtain that $\eta \in \rL^2((A \ovt P)^\omega)$. But this clearly implies that $\xi \in \rL^2((M \ovt P)^\omega)$.
\end{proof}

\begin{proof}[Proof of Theorem \ref{abelian}]
Suppose that $M \ovt N$ is McDuff, i.e.\ $(M \ovt N)_\omega$ is non-commutative. By Proposition \ref{prop_central_1}, we know that $(M \ovt N)_\omega \subset (A \ovt N)_\omega$ so that $(A \ovt N)_\omega$ is also non-commutative. Therefore, we can find $x=(x_n)^\omega$ and $y=(y_n)^\omega$ in $(A \ovt N)_\omega$ with $\| x_n \|_\infty,\| y_n \|_\infty \leq 1$ for all $n$, such that $\| [x,y] \|_2 =\delta > 0$. Let $A=\rL^\infty(T,\mu)$ with $(T,\mu)$ a probability space. Write $x_n=(t \mapsto x_n(t)) \in A \ovt N =\rL^\infty(T,\mu,N)$ with $\| x_n(t) \|_\infty \leq 1$ for all $n$ and $t$. Similarly, let $y_n=(t \mapsto y_n(t))$. Fix $F \subset N$ a finite subset and $\varepsilon > 0$. Since $x,y \in (A \ovt N)_\omega$, we know that
$$ \lim_{n \to \omega} \mu( \{ t \in T \mid \forall a \in F, \: \|[x_n(t),a]\|_2 \leq \varepsilon \})=1$$
and
$$ \lim_{n \to \omega} \mu( \{ t \in T \mid \forall a \in F, \: \|[y_n(t),a]\|_2 \leq \varepsilon \})=1.$$
Moreover, since $\| [x,y] \|_2=\delta > 0$, we have
$$ \lim_{n \to \omega} \mu( \{ t \in T \mid \|[x_n(t),y_n(t)]\|_2 \geq \delta/2 \}) > 0.$$
Hence, for $n$ large enough, the intersection of these three sets is non-empty, i.e.\ there exists $t$ such that 
$$ \forall a \in F, \: \|[x_n(t),a]\|_2 \leq \varepsilon,$$
$$ \forall a \in F, \: \|[y_n(t),a]\|_2 \leq \varepsilon,$$
$$\|[x_n(t),y_n(t)]\|_2 \geq \delta/2.$$
Hence, by iterating this procedure, we can extract a sequence $a_k=x_{n_k}(t_k), \: k \in \N$ and $b_k=y_{n_k}(t_k), \: k \in \N$ such that $a=(a_k)^\omega$ and  $b=(b_k)^\omega$ are in  $N_\omega$ and $\| [a,b] \|_2 \geq \delta/2$. Thus $N_\omega$ is not commutative, i.e.\ $N$ is McDuff as we wanted. 
\end{proof}

The second application is the following lemma which we will need in the proof of Theorem \ref{factor}.
\begin{lem} \label{hypercentral}
Let $M$ and $N$ be finite von Neumann algebras. Then we have 
$$1 \otimes \mathcal{Z}(N_\omega) \subset \mathcal{Z}((1 \otimes N)' \cap (M \ovt N)^\omega).$$
\end{lem}
\begin{proof}
First, we treat the case where $M$ is abelian, i.e.\ $M=\rL^\infty(T,\mu)$ for some probability space $(T,\mu)$. Let $(a_k)_{k \in \N}$ be a $\| \cdot \|_2$-dense sequence in $(N)_1$ and let 
$$ N_k:= \{ x \in (N)_1 \mid \forall r \leq k, \: \| [x,a_r]\|_2 \leq 1/k \}.$$ 

 Let $y=(y_n)^{\omega} \in \mathcal{Z}(N_\omega)$ with $\|y_n\|_\infty$ for all $n$. By \cite[Lemma 10]{McD69}, there exists a sequence of sets $U_k \in \omega, \: k \in \N$ such that
$$\forall k \in \N, \: \forall x \in N_k, \: \forall n \in U_k, \quad \| [y_n,x] \|_2  \leq 1/k.$$

Let $x=(x_n)^\omega \in (1 \otimes N)' \cap (M \ovt N)^\omega$ with $\|x_n\|_\infty \leq 1$ for all $n \in \N$. We want to show that $x(1 \otimes y)=(1\otimes y)x$. Write $x_n=( t \mapsto x_n(t)) \in M \ovt N=\rL^\infty(T,\mu,N)$ with $\| x_n(t) \|_\infty \leq 1$ for all $t$ and all $n \in \N$. Since $x \in (1 \otimes N)' \cap (M \ovt N)^\omega$, there exists a sequence of sets $V_k \in \omega$ such that
$$ \mu( \{ t \in T \mid x_n(t) \in N_k \}) \geq 1-1/k^2 $$ 
for all $n \in V_k$.

Therefore, for all $n \in U_k \cap V_k$, we have
$$ \mu( \{ t \in T \mid \| [y_n,x_n(t)]\|_2 \leq 1/k \}) \geq 1-1/k^2 $$
which implies that
$$ \| [1 \otimes y_n, x_n] \|_2^2 =\int_{T}\| [y_n,x_n(t)]\|_2^2 \: \rd \mu(t)   \leq 5/k^2 .$$
Since $U_k \cap V_k \in \omega$ for all $k \in \N$, we conclude that $\lim_{n \to \omega} \| [1 \otimes y_n, x_n]\|_2 =0$ as we wanted.

Finally, we extend to the general case where $M$ is not necessarily abelian. Let $\xi \in \rL^2((M \ovt N)^\omega)$ be an $N$-central vector. We want to show that $\xi$ is $\mathcal{Z}(N_\omega)$-central. By Lemma \ref{Haagerup}, we know that $\eta=(V \otimes 1)^\omega (\xi) \in \rL^2((A \ovt N)^\omega)$. Since $(V \otimes 1)^\omega$ is $N$-bimodular, we know that $\eta$ is $N$-central.  Hence, by the abelian case, we obtain that $\eta$ is $\mathcal{Z}(N_\omega)$-central. Since $(V \otimes 1)^\omega$ is $N^\omega$-bimodular, we conclude that $\xi$ is also $\mathcal{Z}(N_\omega)$-central.
\end{proof}

\begin{proof}[Proof of Theorem \ref{factor}]
By Lemma \ref{hypercentral}, we know that $1 \otimes \mathcal{Z}(N_\omega)$ is contained in the center of $(1 \otimes N)' \cap (M \ovt N)^\omega$ hence it is also contained in the center of $(M \ovt N)_\omega$. Since $(M \ovt N)_\omega$ is a factor, this implies that $N_\omega$ is also factor, and the same argument shows that $M_\omega$ is a factor. 

Now suppose that $M \ovt N$ is McDuff. Then $(M \ovt N)_\omega$ is non-trivial. Thus $M_\omega$ or $N_\omega$ is also non-trivial (use \cite[Corollary 2.2]{Co75b} or Prososition \ref{prop_central_1}). This means that $M_\omega$ or $N_\omega$ is a non-trivial factor. In particular, $M$ or $N$ is McDuff.
\end{proof}

\begin{proof}[Proof of Corollary \ref{rigid}]
First, we show that $M_\omega$ is a factor. Suppose that $(x_k)_{k \in \N}$ is a non-trivial central sequence in $N$ with $\|x_k\|_2=1$ and $\tau(x_k)=0$ for all $k \in \N$. Then, since $Q_n$ is non-Gamma, we know by \cite[Theorem 2.1]{Co75b} that $\lim_k \|x_k-E_{Q_n' \cap M}(x_k)\|_2=0$ for every $n \in \N$, where
$$ Q_n = M_0 \ovt M_1 \ovt  \cdots \ovt M_n \otimes 1 \otimes 1 \otimes \cdots  \subset M.$$
Hence we can find a sequence $(n_k)_{k \in \N}$ with $n_k \rightarrow \infty$ such that $\| E_{Q_{n_k}' \cap M }(x_k) \| \geq \frac{1}{2}$ for all $k \in \N$. Thus, we can find $y_k$ in the unit ball of $Q_{n_k}' \cap M$ such that $\|[x_k,y_k]\|_2 \geq \frac{1}{4}$ for all $k \in \N$. But, by construction, $(y_k)_{k \in \N}$ is a central sequence in $M$. This shows that $(x_k)_{k \in \N}$ is not hypercentral and therefore that $M_\omega$ is factor. By Theorem \ref{factor}, we conclude that $N_\omega$ is a factor. Moreover, $N$ has property Gamma because of \cite[Theorem 4.1]{Po10}. This means that $N_\omega$ is non-commutative or equivalently that $N$ is McDuff. Thus $N \cong N \ovt R \cong M$ as we wanted.
\end{proof}

\bibliographystyle{plain}

\begin{thebibliography}{MvN43}

%



%


\bibitem[Co75b]{Co75b} {\sc A. Connes}, {\it Classification of injective factors. Cases ${\rm II_1}$, ${\rm II_\infty}$, ${\rm III_\lambda}$, $\lambda \neq 1$.} Ann. of Math. {\bf 74} (1976), 73--115.

\bibitem[Co85]{Co85} {\sc A. Connes}, {\it Factors of type ${\rm III_1}$, property $L'_\lambda$ and closure of inner automorphisms.} J. Operator Theory {\bf 14} (1985), 189--211.

\bibitem[CS76]{CS76} {\sc A. Connes, E. Stormer}, {\it Homogeneity of the state space of factors of type $\III_1$.} J. Funct. Anal. {\bf 28} (1976), 187--196.




\bibitem[FM75]{FM75} {\sc J. Feldman, C.C. Moore}, {\it Ergodic equivalence relations, cohomology, and von Neumann algebras. ${\rm I}$, ${\rm II}$.} Trans. Amer. Math. Soc. {\bf 234} (1977), 289--324, 325--359.

\bibitem[Ha84]{Ha84} {\sc U. Haagerup}, {\it A new proof of the equivalence of injectivity and hyperfiniteness for factors on a separable Hilbert space.} J. Funct. Anal. {\bf 62} (1985), 160--201.

\bibitem[Ha85]{Ha85} {\sc U. Haagerup}, {\it Connes' bicentralizer problem and uniqueness of the injective factor of type ${\rm III_1}$.} Acta Math. {\bf 69} (1986), 95--148.


\bibitem[HMV16]{HMV16} {\sc C. Houdayer, A. Marrakchi, P. Verraedt}, {\it Fullness and Connes' $\tau$ invariant of type ${\rm III}$ tensor product factors.} {\tt arXiv:1611.07914}

\bibitem[HMV17]{HMV17} {\sc C. Houdayer, A. Marrakchi, P. Verraedt}, {\it Strongly ergodic equivalence relations: spectral gap and type $\III$ invariants.} to appear in Ergodic Theory Dynam. Systems.




%
%
%


\bibitem[IV15]{IV15} {\sc A. Ioana, S. Vaes}, {\it Spectral gap for inclusions of von Neumann algebras.} Appendix to the article {\it Cartan subalgebras of amalgamated free product $\II_1$ factors} by {\sc A. Ioana} in Ann. Sci. \'{E}cole Norm. Sup. {\bf 48} (2015), 71--130.


\bibitem[JS87]{JS87} {\sc V.F.R Jones, K. Schmidt}, {\it Asymptotically Invariant Sequences and Approximate Finiteness.} American. J. Math. {\bf 109} (1987), 91--114.


\bibitem[McD69]{McD69} {\sc D. McDuff}, {\it Central sequences and the hyperfinite factor.} Proc. London Math. Soc. {\bf 21} (1970), 443--461.


\bibitem[Ma16]{Ma16} {\sc A. Marrakchi}, {\it Spectral gap characterization of full type $\mathrm{III}$ factors.} {To appear in J. Reine Angew. Math.} {\tt arXiv:1605.09613}

\bibitem[Ma17]{Ma17} {\sc A. Marrakchi}, {\it Strongly ergodic actions have local spectral gap.} {\tt arXiv:1707.00438}

\bibitem[MvN43]{MvN43} {\sc F. Murray, J. von Neumann}, {\it Rings of operators.} ${\rm IV}$.  Ann. of Math. {\bf 44} (1943), 716--808.

%





\bibitem[Po85]{Po85} {\sc S. Popa}, {\it A short proof that injectivity implies hyperfiniteness for finite von Newmann algebras.} J.
Operator Theory {\bf 16} (1986), 261--272.

\bibitem[Po87]{Po87} {\sc S. Popa}, {\it The commutant modulo the set of compact operators of a von Neumann algebra.} J. Funct. Analysis {\bf 712} (1987), 393--408.

\bibitem[Po95]{Po95} {\sc S. Popa}, {\it Free-independent sequences in type $\II_1$ factors and related problems.} Recent advances in
operator algebras. (Orl\'eans 1992) Ast\'erisque No {\bf 232} (1995), 187--202.

\bibitem[Po10]{Po10} {\sc S. Popa}, {\it On spectral gap rigidity and Connes' invariant $\chi(M)$.} Proc. Amer. Math. Soc. {\bf 138} (2010), 3531--3539.

\bibitem[Po14]{Po14} {\sc S. Popa}, {\it Independence properties in subalgebras of ultraproduct $\II_1$ factors.} J. Funct. Analysis {\bf 266} (2014), 5818--5846




%
%





\bibitem[WY14]{WY14} {\sc W. Wu, W. Yuan}, {\it A remark on central sequence algebras
of the tensor product of $\II_1$ factors.} Proc. Amer. Math. Soc. {\bf 142} (2014), 2829--2835.



\end{thebibliography}

\end{document}